\documentclass[12pt,psamsfonts]{amsart}

\usepackage{amsmath,amssymb}
\usepackage{graphicx}
\usepackage[dvips]{psfrag}

\newtheorem{theorem}{Theorem}
\newtheorem{proposition}[theorem]{Proposition}
\newtheorem{corollary}[theorem]{Corollary}
\newtheorem{lemma}[theorem]{Lemma}
\newtheorem {remark}[theorem]{Remark}

\title[Bifurcations in a class of polycycles ]{
Bifurcations in a class of polycycles involving two saddle-nodes
on a M\"obius band}
\author[Claudio Pessoa and Jorge Sotomayor]{}

\thanks{The first author is  supported by FAPESP-Brasil grant 06/56664-0. The second author is fellow of CNPq
and has the partial support of CNPq Grant 473747/2006-5}


\begin{document}
 \maketitle

\centerline{\scshape Claudio Pessoa and Jorge Sotomayor}
\medskip

{\footnotesize \centerline{ Instituto de Matem\'atica e Estat\'\i
stica, Universidade de S\~ao Paulo} \centerline{Rua do Mat\~ao
1010, Cidade Universit\'aria} \centerline{05.508-090, S\~ao Paulo,
SP, Brasil } \centerline{\email{clagp@ime.usp.br,
sotp@ime.usp.br}}}

\medskip

\bigskip

\begin{quote}{\normalfont\fontsize{8}{10}\selectfont
{\bfseries Abstract.} In this  paper we study the bifurcations  of
a class of polycycles, called lips, occurring in generic
three-parameter smooth families of vector fields on a M\"obius
band. The lips consists of a set of polycycles formed by two
saddle-nodes, one attracting and the other repelling, connected by
the hyperbolic separatrices of the saddle-nodes and by orbits
interior to both nodal sectors. We determine, under certain
genericity hypotheses, the maximum number of limits cycles that
may bifurcate from a graphic belonging to the lips and we describe
its bifurcation diagram.
\par}
\end{quote}

\section{Introduction}

\noindent Let $M^2$ be a smooth $2$-dimensional manifold. Let
$X_\mu: M^2\times \Lambda\rightarrow TM^2$ be
 a
  {\it smooth
 family of vector fields on
$M^2$},
depending
 on
 p-parameters, represented by $\mu \in  \Lambda $, where
$\Lambda$ is a neighborhood of the origin in $\mathbb R^p$.

An {\it oriented polycycle}
 of a vector field,
 $X_0$,
  in a 2-dimensional
manifold
 is a cyclically ordered union of singular points,
called {\it vertices} $A_0, \;A_1, \ldots,\; A_n=A_0$ (some of
them may coincide) and different orbits $\gamma_0, \ldots,
\gamma_{n-1}$ endowed with their natural orientation, such that
for any $i=0, \ldots, n-1$ the orbit $\gamma_i=\gamma_i(t)$ tends
to the point $A_i$ as $t\rightarrow -\infty$ and tends to
$A_{i+1}$ as $t\rightarrow +\infty$. These  orbits are called {\it
connections}, or {\it arcs}  of the polycycle.

A {\it saddle-node of multiplicity $2$
 of a vector field
$X_0$
 in $M^2$}
 is a singular
point
 $p_0$
 of
$X_0$
at which its  linearization  has  only one zero eigenvalue and
the restriction of
 $X_0$
  to a center manifold
  has the form
\[
(ax^2+\cdots)\frac{\partial}{\partial x}, \;a\neq 0,\; x\in
\mathbb R,
\]
and the dots denote higher order terms. If the non-zero real
eigenvalue is positive, the saddle-node is called {\it repelling}
and, if it is negative, the saddle-node is called {\it
attracting}.

In the first (resp. second) case the basin of repulsion (resp.
attraction)  of the saddle-node, called the {\it nodal sector}, is
a half-plane bordered by the strong unstable (resp. strong stable)
separatrices and the singularity. Also the basin of attraction
(resp. repulsion) is a curve coincident to half the central
manifold, called the {\it hyperbolic separatrix}, since it is the
common boundary of the  {\it hyperbolic sectors} of the
saddle-node.

We call  {\it lips} the set of polycycles which consists of two
saddle-nodes (one repelling and one attracting) connected by the
hyperbolic separatrices and by orbits interior to both nodal
sectors (see Figure \ref{figura1}).

\begin{figure}[ptb]
\begin{center}
\psfrag{O2}[l][B]{$O_2$} \psfrag{O1}[l][B]{$O_1$}
\psfrag{G2}[l][B]{$\Gamma^+_1$} \psfrag{G1}[l][B]{$\Gamma^-_1$}
\psfrag{G3}[l][B]{$\Gamma^-_2$} \psfrag{G4}[l][B]{ $\Gamma^+_2$}
\psfrag{D1}[l][B]{$\Delta_{\nu,1}$}\psfrag{D2}[l][B]{$\Delta_{\nu,2}$}
\psfrag{F}[l][B]{$f_{\nu}$}\psfrag{G}[l][B]{$g_{\nu}$}
\includegraphics[height=5in,width=5in]{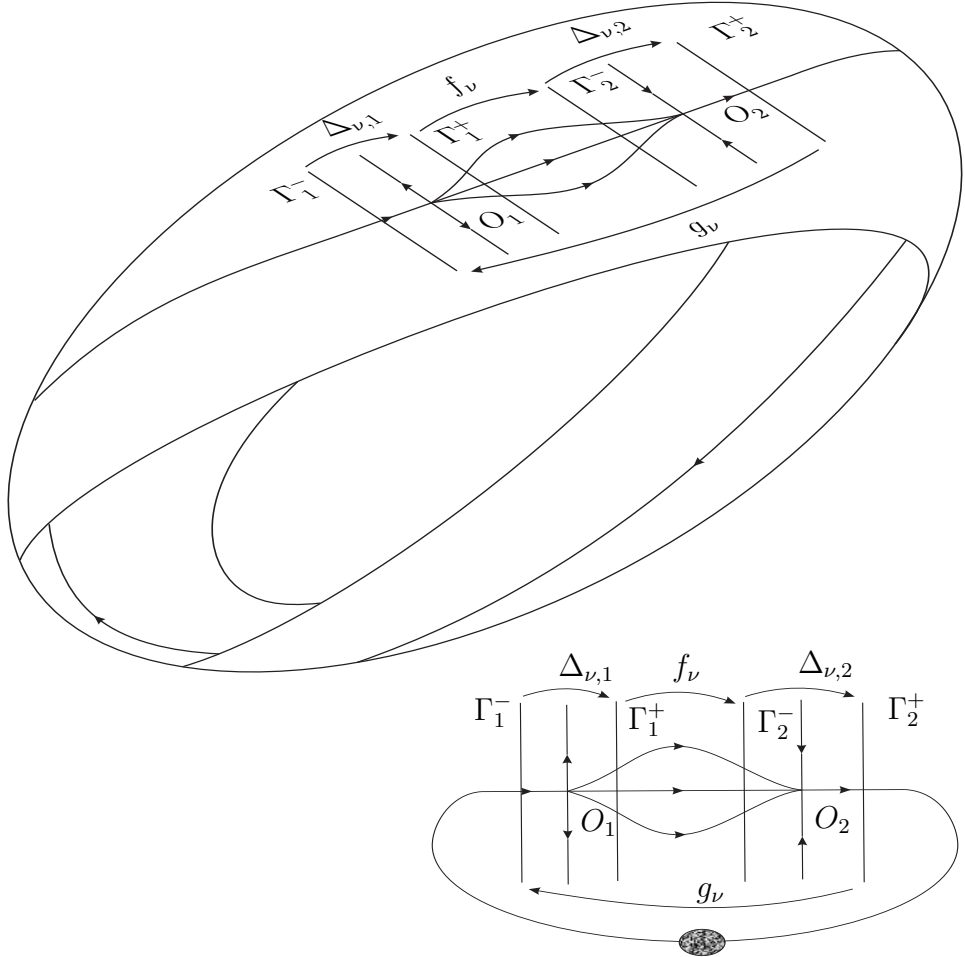}
\end{center}
\caption{Lips on a M\"obius band.} \label{figura1}
\end{figure}

It is know that limit cycles are born, i.e. bifurcate, from limit
periodic sets that are closed invariant subsets of the plane,
eventually containing arcs of nonisolated singularities of vector
fields. But since we study only generic families, with all
singularities isolated, such limit periodic sets can be only
polycycles. The question of the number of limit cycles which can
bifurcate from a polycycle occurring in a generic finite-parameter
family of vector fields, is closely related to the Hilbert
$16^{th}$ Problem (see \cite{RR} for more details).

In recent years several authors have investigated the bifurcation
of limit cycles from a set of polycycles. See, for example,
\cite{DIR}, \cite{GR}, \cite{YL}, \cite{KS} and references
therein.

This work focuses on the
 bifurcation diagram of the
 {\it lips} in the case  where the normal bundle of the
polycycles is non-orientable, i.e. diffeomorphic to a M\"obius
band. The orientable case was studied in \cite{KS}. The {\it
bifurcation diagram} is a stratification of the parameter space
such that to different strata correspond topologically distinct
phase portraits of the vector field.

The set of polycycles of lips type occur persistently in
three-parameter families of vector fields in $M^2$, because we
need two parameters to unfold the saddle-nodes and one parameter
to bifurcate the connection of separatrices of the hyperbolic
sectors of the saddle-nodes. The main tools used to describe the
bifurcation diagram of the lips is a normal form in a neighborhood
of a saddle-node given by Theorem \ref{the:01} (see Section
\ref{sec:02}). This normal form simplifies the problem. In fact,
for the case where there are singular points in the unfolding of a
saddle-node, it is trivial. See Theorem \ref{the:04}. Therefore
the case of interest corresponds to the situation where there are
not singular points, i.e. to the bifurcation of limit cycles.

If $X_\nu(x,y)$, $\nu=(\epsilon,\delta,\lambda)\in \mathbb R^3$,
is a smooth $3$-parameter family of vector fields on $M^2$ such
that for $\nu=0$ $X_0$ has a set of polycycles of lips type, then
for some values of $\nu$ we can define transition maps on the
transversal sections $\Gamma^{\mp}_{1,2}$ (see
Figure~\ref{figura1}). We compose these  maps to obtain the
Poincar\'e map $\Delta_\nu:\Gamma^+_1\rightarrow \Gamma^+_1$.
Thus, the study of the bifurcations of limit cycles from the lips
is reduced to the investigation of the bifurcations of periodic
points of  $3$-parameter families of real maps. In fact, when
studying the bifurcations of limit cycles in a M\"obius band, we
deal with bifurcations of fixed points and points of period $2$ of
the Poincar\'e map and their variations in function of the
parameters. Isolated fixed points (resp. 2-periodic points) are in
one-to-one correspondence with limit cycles that appear after
perturbation of the lips and intersect $\Gamma^+_1$ only once
(resp. twice). Moreover, simple fixed and periodic points
correspond to hyperbolic cycles, double fixed points correspond to
semistable cycles, etc.

There are two possible types of bifurcations:
\begin{itemize}
\item[(1)] splitting of multiple fixed and periodic points, and
\item[(2)] escaping of fixed and periodic points through the
boundary points of $\Gamma^+_1$.
\end{itemize}

Therefore the bifurcation surface of the Poincar\'e map, i.e. the
surface in the parameter space where the number of periodic and
fixed points change, is the union of four surfaces, $\Sigma_1$,
$\Sigma_2$ and $\Sigma_\pm$. On $\Sigma_1$ we have multiple fixed
points,
and on $\Sigma_2$, we have multiple $2$-periodic points. Now, on
$\Sigma_+$ (resp. $\Sigma_-$) there is at least one periodic point
equal to the positive extreme of $\Gamma^+_1$ (resp. the negative
extreme of $\Gamma^+_1$). In terms of bifurcations of
 the
original system, the union of the surfaces $\Sigma_1$
 and $\Sigma_2 $ corresponds to the splitting of  multiple limit
cycles, while the union $\Sigma_+\cup \Sigma_-$ corresponds to
cycles escaping from the domain where the system is considered.

We will see in Section \ref{sec:05} that to determine the periodic
points of the Poincar\'e map is equivalent to determine the roots
of an  equation
\begin{equation}
\label{eq:30} \varphi_\delta(y,p,q)=0,
\end{equation}
where the parameters $(\delta,p,q)$ are introduced by the blow-up
\linebreak$(\epsilon,\delta,\lambda)\mapsto
\Phi(\epsilon,\delta,\lambda)=(\delta,p,q)$ defined in Section
\ref{sec:03}. This blown-up takes the point $\nu=0$ into the half
plane $\mathbb R^3_{(\delta,p,q)}\cap\{p>0,\;\mbox{\rm
and}\;\delta=0\}$.  The function in equation (\ref{eq:30}) is
essentially $\Delta_\nu -y$, in the blown-up coordinates.

{For} each fixed $\delta$ we have  a surface $S_\delta\in \mathbb
R^3_{(y,p,q)}$ determined by equation \eqref{eq:30}. As we are
interested in the multiple roots of equation \eqref{eq:30}, we
must characterize the projection in the $(p,q)$-plane of the curve
$C_\delta$ given by equations
\[
\varphi_\delta(y,p,q)=0,\;\;\frac{\partial
\varphi_\delta}{\partial y}(y,p,q)=0.
\]
The characterization of this curve is given in Sections
\ref{sec:04} and \ref{sec:05}. In the parameter space $\mathbb
R^3_{(\delta,p,q)}$, the surface $\Sigma=\Sigma_1\cup\Sigma_2$ is
described by the property that its intersection with the plane
$\delta$=constant is equal to the trace of the projection of the
curve $C_\delta$ in the $(p,q)$-plane. The main result of this
work, Theorem \ref{the:08} of Section \ref{sec:06}, formalizes
this discussion and shows that $\Sigma=\Sigma_1\cup\Sigma_2$ in
$\mathbb R^3_{(\delta,p,q)}$ is diffeomorphic to the cylinder over
$\Lambda_0$ the projection of $C_0$ in the $(p,q)$-plane with axis
parallel to $\delta$ axis (see Figure~\ref{figura03}).
Analogously, we characterize the surfaces $\Sigma_\pm$, see
Theorem \ref{the:08}.


The results in this paper  can be regarded as an extension to
 a
M\"obius band of the work of Kotova and Stanzo \cite{KS} carried
out for the orientable case, that is when the normal bundle of the
polycycles are  diffeomorphic to a cylinder. The main difference
between the two cases  it is that in the nonorientable one we have
to study also the $2$-periodic points of the Poincar\'e map as
well as the flip fixed points (i.e with negative derivative)
corresponding to the one-sided periodic orbits. The nonorientable
case has the additional complication  of presenting  the flip
bifurcations  codimensions $1$ and $2$ (see \cite{BV} and
\cite{JG}).

This paper is organized as follows. Section $2$ is devoted to
review some preliminary results pertinent to the saddle-nodes and
their normal forms. Section $3$ describes the bifurcation diagram
when there is at least one singularity in the unfolding of the
lips. In Section $4$ is studied a change of coordinates which
greatly simplifies the expression of the return map.
Section $6$ is devoted to the definition of the blow-up $\Phi$ in
the parameters, opening the origin to the $(p,q)$-plane. In
Sections $7$ and $8$  are  characterized the bifurcation surfaces
$\Sigma_1$ and $\Sigma_2$. The main theorem of this paper,
presenting a synthesis of the bifurcation diagram is proved in
Section $9$ (Theorem \ref{the:08}). Section $10$ studies the
cyclicity of an individual  polycycle of the lips (Theorem 17).

\section{Preliminaries}
\label{sec:02}

In this section we present the tools to study the bifurcation
diagram of the lips.

Let $X_\mu(x,y)$, $\mu=(\mu_1,\mu_2,\mu_3)\in \mathbb R^3$ be a
smooth $3$-parameter family of vector fields on $M^2$ such that
for $\mu=0$, $X_0$ has a set of polycycles of lips type.

We denote by $O_1$ and $O_2$ the saddle-nodes of $X_0(x,y)$. In
what follows we assume that the orientation is chosen as in Figure
\ref{figura1}. As we can see between the nodal sectors of
saddle-nodes $O_1$ and $O_2$ form a region filled out by arcs of
the lips. On the topological boundary of this region, other types
of separatrix  connections  involving singular points outside the
lips may occur. In this paper, however, we  do not investigate
effects here caused by this circumstance at the boundary; instead
we choose a transversal section intersecting certain connections
and consider their saturation by the flow curves, assuming that
all of them are also connections of the lips. When the parameters
of the family change, we consider only the orbits intersecting the
same transversal section, which is assumed to be independent of
the parameters. This restriction produces the bifurcations when a
limit cycle leaves the domain under consideration.

To describe the bifurcation diagram of the lips we need the
following result that is proved in \cite{YL}.

\begin{theorem}
\label{the:01} Let $X_\mu$, $\mu\in \mathbb R^3$, be a generic
$3$-parameters smooth family of vector fields in a two manifold
$M^2$, such that $X_0$ has a saddle-node of multiplicity $2$ in
the origin $(0,0)$. Then the family $X_\mu$ may be reduced by a
finitely smooth change of coordinates, time rescaling and
parameter change to the following normal forms
\begin{equation}
\begin{array}{lclr}
\dot{x} & = & (x^2+\epsilon)(1+a(\mu)x)^{-1}, & \epsilon =\epsilon(\mu)\\
\dot{y} & = & \pm y.
\end{array}
\end{equation}

\end{theorem}

Let us give the explicit genericity assumption for the family in
the previous theorem. The family $X_\mu$ intersects transversally
at the point $X_0$ the surface of vector fields having degenerate
singular points. This means that if we consider the center
manifold of the local family $X_\mu$ and write the restriction of
the family to this manifold as
\[
\dot{x}=f(x,\mu), \; \dot{\mu}=0,\; (x,\mu)\in \mathbb R\times
\mathbb R^3,\; f(0,0)=0,
\]
then
\[
\left.\frac{\partial f(0,\mu)}{\partial \mu}\right|_{\mu= 0}\neq
0.
\]

The smoothness of the normalizing chart in the phase variables and
the parameters in the previous theorem is arbitrarily high, but
finite; the smoothness may be increased after shrinking the domain
of the normalizing chart.

In order to investigate the bifurcations of the lips, we introduce
the system of parameters $\nu=(\epsilon,\delta,\lambda)$. We will
suppose that the hyperbolic eigenvalues of saddle-nodes $O_1$ and
$O_2$ are real and have opposite signs. If
$\mu=(\mu_1,\mu_2,\mu_3)$ are the original parameters and the lips
occur for $\mu=0$, then, according to Theorem \ref{the:01}, in a
neighborhood of the saddle-node $O_1$ there exist local
coordinates in which the family of vector fields has the form
\begin{equation}
\label{eq:02}
\begin{array}{lclr}
\dot{x} & = & (x^2+\epsilon)(1+a_1(\mu)x)^{-1}, & \\
\dot{y} & = & y, & \epsilon=\epsilon(\mu).
\end{array}
\end{equation}

Let $\Gamma^\pm_{1}$ be two transversal sections to the flow that
are given in the canonical chart (i.e where expression \ref{eq:02}
holds). Without loss of generality we can suppose that
\[
\Gamma^\pm_{1}=\{(x,y):x=\pm 1,|y|\leq 1\}.
\]
The transversal section $\Gamma^-$ is the entrance gate: all
orbits crossing it enter the neighborhood of the singularity
$O_1$. The other section is the exit gate in the same sense.
Clearly, for $\epsilon
>0$ the derivative $\dot{x}$ is positive; therefore each orbit
starting on $\Gamma^-_1$ will intersect $\Gamma^+_1$ at a certain
point; we denote this transition map by
\[
\Delta_{\nu,1}:\Gamma^-_1\rightarrow \Gamma^+_1.
\]

In the same manner, there exist a normalizing chart around $O_2$,
in which the family has the form
\begin{equation}
\label{eq:03}
\begin{array}{lclr}
\dot{x} & = & (x^2+\delta)(1+a_2(\mu)x)^{-1}, & \\
\dot{y} & = & -y, & \delta=\delta(\mu).
\end{array}
\end{equation}

We take the transversal sections
\[
\Gamma^\pm_2=\{(x,y):x=\pm 1,|y|\leq 1\},
\]
and similarly as in the previous case we obtain the transition map
\[
\Delta_{\nu,2}:\Gamma^-_2\rightarrow \Gamma^+_2.
\]
Note that this map is defined only for $\delta >0$, although the
transversal sections $\Gamma^{\pm}_2$ are well defined for all
small $\delta$.

Besides the two maps $\Delta_{\nu,1}$, $\Delta_{\nu,2}$, there are
two regular transition maps along connections depending on
parameters. In the normalizing charts they can be written as
\begin{equation}
\label{eq:01} f_\nu:\Gamma^+_1\rightarrow
\Gamma^-_2,\;\;\;g_\nu:\Gamma^+_2\rightarrow \Gamma^-_1.
\end{equation}
Since for $\mu=0$ the points $O_i$ are connected along the
hyperbolic separatrices, i.e. $g_0(0)=0$ and $g_{\nu}'(0)<0$
(orientation reversing). Denote by $\lambda$ the relative
displacement of the separatrices (see Figure \ref{figura1}),
\[
\lambda=g_\nu(0).
\]

In order to proceed further, we need an additional genericity
assumptions. We will require, from now on, that the {\it Jacobian
of the map} $(\mu_1,\mu_2,$ $\mu_3)\mapsto
(\epsilon,\delta,\lambda)$ to be nonvanishing,
\[
\det \left(\left.\frac{\partial\nu}{\partial
\mu}\right|_{\mu=0}\right)\neq 0.
\]

If this condition is satisfied, we will describe the bifurcation
diagram for the lips in terms of the new parameters $\nu$ rather
than $\mu$: the above genericity assumption guarantees that the
bifurcation diagram in the original parameter space will be
diffeomorphic to the one obtained in Theorem \ref{the:08}.

Now we will determine expressions for the maps $\Delta_{\nu,1}$,
$\Delta_{\nu,2}$ defined above. In fact, by explicit integration
of the normal forms, which has separated variables, it follows
that
\begin{equation}
\label{eq:09} \Delta_{\nu,1}(y)=C_1(\epsilon)^{-1}y \;\mbox{and}\;
\Delta_{\nu,2}=C_2(\delta)y,
\end{equation}
with $C_1(\epsilon)\rightarrow 0$ when $\epsilon\rightarrow 0$ and
$C_2(\delta)\rightarrow 0$ when $\delta\rightarrow 0$. More
precisely,
\begin{equation}
\label{eq:19} C_1(\epsilon)=\exp(-\frac{2}{\sqrt{\epsilon}}\arctan
\frac{1}{\sqrt{\epsilon}})\;\mbox{and}\;C_2(\delta)=\exp(-\frac{2}{\sqrt{\delta}}\arctan
\frac{1}{\sqrt{\delta}}).
\end{equation}

\section{Description of the bifurcation diagram outside the positive quadrant $\epsilon>0, \delta>0$}

The complete description of the bifurcation diagram in the domain
of parameters where there is at least one singular point of the
vector field, is given by the following theorem.

\begin{theorem}
\label{the:04}
The bifurcation diagram in the intersection of a
small neighborhood of the origin in the space $\mathbb
R^{3}_{(\epsilon, \delta, \lambda )}$ with the set  $\{\epsilon
\leq 0\}\cup \{\delta \leq 0\}$, consists of twelve components,
corresponding to topologically nonequivalent phase portraits
differing by the type of singular points and the existence of
connections between them:
\begin{enumerate}
\item $\epsilon =\delta =\lambda =0$, two saddle-nodes connected
by a separatrix,

\item $\epsilon =\delta =0$, $\lambda \neq 0$, two saddle-nodes
without connection,

\item $\epsilon <0$, $\delta <0$, $\lambda =0$, two saddles
connected by a separatrix, one stable and one unstable node,

\item $\epsilon <0$, $\delta <0$, $\lambda \neq 0$, two saddles
without connection, a stable and an unstable node,

\item $\epsilon <0$, $\delta =0$, $\lambda =0$, saddle and
saddle-node connected by a separatrix, and also a stable node,

\item $\epsilon <0$, $\delta =0$, $\lambda \neq 0$, saddle,
saddle-node and a stable node without connection,

\item $\epsilon <0$, $\delta
>0$, a saddle and a stable node,

\item $\epsilon =0$, $\delta <0$, $\lambda =0$, a saddle and a
saddle-node connected by a separatrix, and also an unstable node,

\item $\epsilon =0$, $\delta <0$, $\lambda \neq 0$, a saddle, a
saddle-node and an unstable node without connections,

\item $\epsilon =0$, $\delta
>0$ or $\epsilon
>0$, $\delta =0$, a saddle-node,

\item $\epsilon >0$, $\delta <0$, saddle and an unstable node.
\end{enumerate}
\end{theorem}
The proof immediately follows from the local normal forms near the
singularities $O_{1,2}$.

\section{Admissible change of coordinates}

A {\it $k$-admissible change of coordinates} is a local family of
\linebreak$C^k$-diffeomorphisms
\[
\Phi_\mu(x,y)=(\xi_\mu(x,y),\zeta_\mu(x,y))
\]
such that
\begin{itemize}
\item[(1)] $\Phi_\mu$ is defined in a connected neighborhood
$U\subset\mathbb R^2$ of the origin;

\item[(2)] $\Phi_\mu$ preserve the axes;

\item[(3)] $\Phi_\mu$ preserve the normal form obtained in Theorem
\ref{the:01};

\item[(4)] the map $(x,y,\mu)\mapsto \Phi_\mu(x,y)$ is of class
$C^k$.
\end{itemize}

Let $\Phi_{\nu,1}$ be a $k$-admissible change of coordinates
defined in a neighborhood of the saddle-node $O_1$. Assume that
the transversal sections $\Gamma^\mp_1$ belong both to the domain
and to the image of $\Phi_{\nu,1}$. We have that $\Phi_{\nu,1}$
generates two local $C^k$ diffeomorphisms, $\varphi_{\nu,1}(y)$,
$\psi_{\nu,1}(y)$, such that if $\tilde{\Delta}_{\nu,1}$ is the
transition map from $\Gamma^-_1$ to $\Gamma^+_1$ in the
coordinates $(\xi_{\nu,1},\zeta_{\nu,1})$, which are defined for
$\epsilon
>0$, then
\begin{equation}
\label{eq:04} \tilde{\Delta}_{\nu,1}\circ
\varphi_{\nu,1}=\psi_{\nu,1}\circ \Delta_{\nu,1},
\end{equation}
where $\Delta_{\nu,1}$ is the transition map from $\Gamma^-_1$ to
$\Gamma^+_1$ in the coordinates $(x,y)$. More precisely,
$\varphi_{\nu,1}$ represents the transition function from
$\Gamma^-_1=\{(x,y):x=-1,|y|\leq 1\}$ to the section
$\Gamma^-_1=\{(\xi_{\nu,1},\zeta_{\nu,1}):\xi_{\nu,1}=-1,|\zeta_{\nu,1}|\leq
1\}$.

In the same way, $\psi_{\nu,1}$ represents the transition function
from $\Gamma^+_1=\{(x,y):x=1,|y|\leq 1\}$ to the section
$\Gamma^+_1=\{(\xi_{\nu,1},\zeta_{\nu,1}):\xi_{\nu,1}=1,|\zeta_{\nu,1}|\leq
1\}$.

We call the $C^k$ local family of diffeomorphism
$\varphi_{\nu,1}$, obtained above, the {\it entrance family
associated to the saddle-node $O_1$} and we call {\it exit family
associated to the saddle-node $O_1$}, the other $C^k$ local family
of diffeomorphisms $\psi_{\nu,1}$, obtained above.

In the same way, we denote by $\varphi_{\nu,2}$ the {\it entrance
family associated to the saddle-node $O_2$} and by $\psi_{\nu,2}$
the {\it exit family associated to the saddle-node $O_2$}, defined
by a $k$-admissible change of coordinates $\Phi_{\nu,2}$  in a
neighborhood of saddle-node $O_2$.

\begin{proposition}
\label{pro:01} Let $\varphi_{\nu,1}$, $\psi_{\nu,2}$, be $C^k$
local increasing families of diffeomorphisms such that
$\varphi_{\nu,1}(0)=\psi_{\nu,2}(0)=0$. Then there exist
$k$-admissible changes of coordinates $\Phi_{\nu,i}$, $i=1,2$,
such that the entrance family associated to saddle-node $O_1$ and
the exit family associated to saddle-node $O_2$ are
$\varphi_{\nu,1}$, and $\psi_{\nu,2}$, respectively.
\end{proposition}

The proof of this proposition can be found in \cite{GR} page $53$.

\begin{theorem}
There exist k-admissible changes of coordinates, such that in
these coordinates the map $g_\nu$ can be written in the form
\begin{equation}
\label{eq:08} g_{\nu}(y)=-y+\lambda.
\end{equation}
\end{theorem}
\begin{proof}
Let $\varphi_{\nu,1}$ and $\psi_{\nu,2}$ be families of
diffeomorphisms satisfying the hypothesis of Proposition
\ref{pro:01}. Then, there exist $k$-admissible changes of
coordinates $\Phi_{\nu,1}(x,y)= (\xi_{\nu,1},\zeta_{\nu,1})$ and
$\Phi_{\nu,2}(x,y)= (\xi_{\nu,2},\zeta_{\nu,2})$ such that
$\varphi_{\nu,1}$ is the entrance family associated to saddle-node
$O_1$ and $\psi_{\nu,2}$ is the exit family associated to
saddle-node $O_2$. In admissible coordinates the transition map
$g_\nu :\Gamma^+_2\rightarrow \Gamma^-_1$, defined in
\eqref{eq:01}, is given by
\[
\tilde{g}_\nu=\varphi_{\nu,1}\circ g_\nu \circ
(\psi_{\nu,2})^{-1},
\]
i.e.,
\[
\tilde{g}_\nu\circ \psi_{\nu,2}=\varphi_{\nu,1}\circ g_\nu.
\]
We want to obtain that $\tilde{g}_\nu
(\zeta_{\nu,2})=-\zeta_{\nu,2}+\lambda$, where $g_\nu(0)=\lambda$.
Therefore, by the expression above, we have to choose
$\varphi_{\nu,1}$ and $\psi_{\nu,2}$ such that the equation
\[
-\psi_{\nu,2}(y)+g_\nu(0) = \varphi_{\nu,1}(g_{\nu}(y))
\]
has a solution. Hence, it is sufficient to take
$\varphi_{\nu,1}(y)=y$ and $\psi_{\nu,2}(y)=-g_\nu(y)+g_\nu(0)$.

\end{proof}

\section{Limit cycles bifurcating from the lips} \label{sec:01}

In the non-trivial part of the bifurcation diagram, where there
are no singular points, different strata of the diagram correspond
to different numbers and position of limit cycles. From this point
on, it is helpful to have in mind Figure~\ref{figura1} and the
notation in Section \ref{sec:02}.

The Poincar\'e map $\Delta_\nu:\Gamma^+_1\rightarrow\Gamma^+_1$
for each $\nu$ is the composition,
\[
\Delta_\nu=f_\nu^{-1}\circ \Delta_{\nu,2}^{-1}\circ
g_\nu^{-1}\circ \Delta_{\nu,1}^{-1},
\]
considered as the composition of one-dimensional maps depending on
the parameters. Note that in the M\"obius band the limit cycles
that bifurcate from the lips correspond to the fixed points of
periods $1$ and $2$ of $\Delta_\nu$. Therefore, the equations that
determine  limit cycles are
\[
\Delta_\nu(y)=y\;\mbox{and}\;\Delta_\nu^2(y)=y.
\]
In fact, the solutions of $\Delta_\nu(y)=y$ are contained in the
set of solutions of $\Delta_\nu^2(y)=y$. The equations of first
and second return fixed points can be rewritten in the form
\begin{equation}
\label{eq:07}
\begin{array}{rcl}
\Delta_{\nu,2}^{-1}\circ g_\nu^{-1}\circ
\Delta_{\nu,1}^{-1}(y) & = & f_\nu(y), \\
f_\nu^{-1}\circ\Delta_{\nu,2}^{-1}\circ g_\nu^{-1}\circ
\Delta_{\nu,1}^{-1}(y) & = & \Delta_{\nu,1}\circ g_\nu\circ
\Delta_{\nu,2}\circ f_\nu(y).
\end{array}
\end{equation}
Introducing the new parmeters
\[
p=\frac{C_1(\epsilon)}{C_2(\delta)} \;\,\, \mbox{and}\;\,\,
q=\frac{\lambda}{C_2(\delta)},
\]
by expressions \eqref{eq:09}, \eqref{eq:19} and \eqref{eq:08},
\eqref{eq:07} becomes

\begin{eqnarray}
-py+q & = & f(y)+r_1(y,\nu),\label{eq:10} \\
-py+q & = &
f\left(-\frac{1}{p}f(y)+\frac{q}{p}\right)+r_2(y,\nu),\label{eq:11}
\end{eqnarray}
where
\[
f(y)=f_0(y),\,\, r_1(y,\nu)=f_\nu(y)-f_0(y)
\]
and
\[
r_2(y,\nu)=f_\nu\left(-\frac{1}{p}f_\nu(y)+\frac{q}{p}\right)-f_0\left(-\frac{1}{p}f_0(y)+\frac{q}{p}\right).
\]
Note that $r_1(y,\nu)\rightarrow 0$ and $r_2(y,\nu)\rightarrow 0$
in $C^k-norm$ on $[-1,1]$ when $\nu\rightarrow 0$, since $f_\nu$
is $C^k$-smooth.

As in the previous section, the transversal section $\Gamma^+_1$
is simply the interval $[-1,1]$. Thus we need to investigate the
equations \eqref{eq:10} and \eqref{eq:11} for $y\in [-1,1]$. When
studying the bifurcations of limit cycles, we in fact deal with
bifurcations of roots of \eqref{eq:11} on $[-1,1]$, and their
variations in function of the parameters. Isolated roots of
equation \eqref{eq:11} on the segment $[-1,1]$ are in one-to-one
correspondence with limit cycles intersecting $\Gamma^+_1$ that
appear after perturbation of the lips. Moreover, simple roots
correspond to hyperbolic cycles, double roots correspond to
semistable cycles, etc.

There are two possible types of bifurcations:
\begin{itemize}
\item[(1)] splitting of multiple roots, and \item[(2)] escaping of
a root through the boundary points of $\Gamma^+_1$.
\end{itemize}

Therefore the bifurcation surface of equation \eqref{eq:11}, i.e.
the surface in the parameter space where the number of roots
change, is the union of four surfaces, $\Sigma_1$, $\Sigma_2$ and
$\Sigma_\pm$. On $\Sigma_1$ we have roots of \eqref{eq:10} that
are multiple roots of \eqref{eq:11} and on $\Sigma_2$ we have
multiple roots of \eqref{eq:11}. Now, on $\Sigma_+$ (resp.
$\Sigma_-$) there is at least one root equal to $1$ (resp. $-1$).

In terms of bifurcations in the original system, the surface
$\Sigma_1\cup \Sigma_2 $ corresponds to the splitting of a
multiple limit cycles, while the union $\Sigma_+\cup \Sigma_-$
correspond to cycles escaping from the domain where the system is
considered.

Fix $P>\max_{y\in [-1,1]}f'(y)$ and consider the subset in
$\mathbb R^3_\nu$ for which $0<p<P$. In this domain the equations
\eqref{eq:10} and \eqref{eq:11} can be regarded as small
perturbations of the equations
\begin{eqnarray}
-py+q & = & f(y),\label{eq:12} \\
-py+q & = &
f\left(-\frac{1}{p}f(y)+\frac{q}{p}\right).\label{eq:13}
\end{eqnarray}

%

\section{The second reparametrization}
\label{sec:03} Let $V\subset \mathbb R^3_\nu$ be a small
neighborhood of the origin and, denote by $V^+$ the intersection
$V\cap \{\epsilon>0, \delta>0\}$. Without loss of generality we
may assume that $V^+$ is a small cube with the edges parallel to
the coordinate axes.

Consider the reparametrization map $\Phi:V^+\rightarrow \mathbb
R^{3+}_{(\delta, p, q)}=\mathbb R^{3}_{(\delta, p, q)}\cap\{p>0\}$
defined by the formula
\[
(\delta,\epsilon,\lambda)\mapsto \Phi
(\delta,\epsilon,\lambda)=(\delta,p,q)=\left(\delta,\frac{C_1(\epsilon)}{C_2(\delta)},\frac{\lambda}{C_2(\delta)}\right).
\]

\begin{lemma}
\label{lem:01} The map $\Phi$ defined above has the following
properties:
\begin{itemize}
\item[(a)] The domain $\Phi(V^+)$ is unbounded. The half plane
$poq^+=$ \linebreak$\mathbb R^{3+}_{(\delta,p,q)}\cap
\{\delta=0\}$ belongs to the boundary of $\Phi(V^+)$.

\item[(b)] The inverse map $\Phi^{-1}$ is defined on $\Phi(V^+)$
and extends by continuity to $poq^+$. The extended map is
continuously differentiable in $\delta$ at $\delta =0$ and takes
the half plane $poq^+$ into the point $\nu=0$.

\item[(c)] For any compact set $D\subset \mathbb R^{2+}_{(p,q)}$
there exist $\delta_D>0$ such that the cylinder
$Z_D=(0,\delta_D)\times D$ belongs to $\Phi(V^+)$.

\item[(d)] For any point $A\in \mathbb R^{2+}_{(p,q)}$ the
extended map $\Phi^{-1}$ takes the semi-interval
$[0,\delta_A)\times A$ into a part of a curve tangent at $\nu=0$
to the line $\delta =\epsilon$, $\lambda=0$.
\end{itemize}
\end{lemma}

The proof of this lema can be found in \cite{KS} page 186.

\begin{corollary}
{For} any compact set $D\subset \mathbb R^{2+}_{(p,q)}$, the map
$\Phi^{-1}$ takes the cylinder $Z_D$ into a narrow horn with
vertex at $\nu=0$, tangent to the line $\delta =\epsilon$,
$\lambda=0$; the projection of this horn to the plane $\delta
o\epsilon$ has an opening of the order $~\delta^{\frac{3}{2}}$,
and the projection to the plane $\delta o\lambda$ has an
exponentially small opening of the order
$~\exp\left(-\frac{\pi}{\sqrt{\delta}}\right)$.
\end{corollary}

We say that the horn $\Phi^{-1}$ correspond to the compact set
$D$.

\begin{corollary}
There exists a compact set $D\subset R^{2+}_{(p,q)}$ such that the
surfaces $\Sigma_1$ and $\Sigma_2$ lie inside the corresponding
horn.
\end{corollary}
\begin{proof}
The equation \eqref{eq:11} has multiple roots if its roots satisfy
the equation
\[
p^2=f'\left(-\frac{1}{p}f(y)+\frac{q}{p}\right)f'(y)+
(r_2)'_y(y,\nu).
\]

Now, as $f$ is a smooth increasing diffeomorphism, we have that
\linebreak $\min_{y\in[-1,1]}f'(y)>0$. Hence, fix positive
constants
\[
P_0<\min_{y\in[-1,1]}f'(y), \; P_1>\max_{y\in[-1,1]}f'(y)
\]
and
\[
Q>\max_{y\in[-1,1]}f'(y)+\max_{y\in[-1,1]}|f(y)|.
\]
Choose $D=[P_0,P_1]\times [-Q,Q]$. Then for $(p,q)\not\in D$ the
equation \eqref{eq:11} has only simple roots. Hence, $\Sigma_1\cup
\Sigma_2\subset \Phi^{-1}(Z_D)$.
\end{proof}

Using Lemma \ref{lem:01} we can perform a new parametrization
$(\delta, \epsilon, \lambda)\mapsto (\delta,p,q)$ so that
equations \eqref{eq:10}, \eqref{eq:11} become
\begin{eqnarray}
-py+q & = & f(y)+\tilde{r}_1(y,\delta,p,q),\label{eq:14} \\
-py+q & = &
f\left(-\frac{1}{p}f(y)+\frac{q}{p}\right)+\tilde{r}_2(y,\delta,p,q),\label{eq:15}
\end{eqnarray}
where $\tilde{r}_i(y,\delta,p,q)=r_i(y,\Phi^{-1}(\delta,p,q))$,
$i=1,2$.

From Lemma \ref{lem:01} and the properties of $r_i$ the following
properties of $\tilde{r}_i$ easily follow.

\begin{lemma}
The functions $\tilde{r}_i$ can be extended continuously to the
half plane $poq^+$, such that $\tilde{r}_i(x,0,p,q)=0$ and for
$\delta >0$, $\tilde{r}_i$ is of the class $C^k$ on $\Phi(V^+)$
and its extension to the half plane $poq^+$ is of class $C^1$.
Moreover, when $\delta \rightarrow 0$, any partial derivatives of
$\tilde{r}_i$, of order less than or equal to $k$, converge
uniformly to $0$ in $[-1,1]\times [P_0,P_1]\times[-Q,Q]$.
\end{lemma}

\section{Characterization of surface $\Sigma_1$}
\label{sec:04}

From the previous section, to characterize the surfaces $\Sigma_1$
and $\Sigma_2$, we need to study the equations \eqref{eq:14} and
\eqref{eq:15}. First, we give the description of surface
$\Sigma_1$ in the space of parameters $(\delta, p, q)$.

The surface $\Sigma_1$ on the space of parameters $\mathbb
R^3_{(\delta, p, q)}$ is the set of points $(\delta, p, q)$ such
that to these values of parameters the fixed points of the fist
return $\Delta_\nu (y) $, or equivalently the roots of
\eqref{eq:14}, are multiple roots of $\Delta_\nu^2(y)-y$, or
equivalently are multiple roots of equation \eqref{eq:15}. Hence,
to determine $\Sigma_1$ we have to study the following system of
equations
\[
\begin{array}{lcl}
\Delta_\nu (y)-y & = & 0, \\ \displaystyle \frac{\partial
\Delta_\nu^2}{\partial y}(y)-1 & = & 0,
\end{array}
\]
or the equivalently (by the chain rule) the system
\begin{equation}
\label{eq:16}
\begin{array}{lcl}
\Delta_\nu (y)-y & = & 0, \\ \displaystyle \frac{\partial
\Delta_\nu}{\partial y}(y)+1 & = & 0.
\end{array}
\end{equation}
Note that the roots of $\Delta_\nu (y)-y$ are all simple.
Moreover, for each value of the parameters $(\delta,p,q)$
corresponds a unique root of $\Delta_\nu (y)-y$, i.e the
Poincar\'e map has only one fixed point for each value of the
parameters $(\delta,p,q)$. In fact, if $y$ is a multiple root of
$\Delta_\nu (y)-y$, i.e. $\displaystyle \frac{\partial
\Delta_\nu}{\partial y}(y)=1$, then by \eqref{eq:14} we have that
$f'(y)=-p-(\tilde{r}_{1})'_y(y,\delta,p,q)$, but this is a
contradiction, because $f'(y)>0$, $p>0$ and
$(\tilde{r}_{1})'_y\rightarrow 0$, when $\delta\rightarrow 0$.
Now, as $\Delta_\nu (y)=f^{-1}_{\Phi^{-1}(\delta,p,q)}(-py+q)$, it
follows from \eqref{eq:14} that system \eqref{eq:16} becomes
\begin{equation}
\label{eq:17}
\begin{array}{rcl}
-py+q & = & f(y)+\tilde{r}_1(y,\delta,p,q), \\ \displaystyle
f'(y)+ (\tilde{r}_{1})'_y(y,\delta,p,q)& = & p.
\end{array}
\end{equation}

Define the function $\phi_\delta$ by
\[
\phi_\delta (y,p,q)=-py+q - f(y)-\tilde{r}_1(y,\delta,p,q).
\]
For $\delta =0$, $\tilde{r}_1(y,0,p,q)=0$, and as the gradient of
$\phi_0$ does not vanish, it follows that $\phi_0=0$ is a regular
surface in $\mathbb R^3_{(y,p,q)}$. Hence, for $\delta$
sufficiently small, $\phi_\delta=0$ is also a regular surface in
$\mathbb R^3_{(y,p,q)}$. Therefore, for $\delta$ fix small enough,
system \eqref{eq:17} determine a curve $C_\delta$ on $\mathbb
R^3_{(y,p,q)}$ and the projection $(y,p,q)\mapsto (p,q)$ is a
curve $L_\delta$ in $\mathbb R^2_{(p,q)}$ without singular points
and self-intersections, because $(\phi_{\delta})'_y(y,p,q)\neq 0$.
In fact, for $\delta =0$, from \eqref{eq:17} it follows that
\begin{equation}
\label{eq:18} L_0(y)=(f'(y), f'(y)y+f(y)).
\end{equation}

The following result is a straightforward consequence of the
previous arguments.

\begin{theorem}
\label{the:03} The surface $\Sigma_1$ is a horn in $\mathbb
R^3_{(\epsilon, \delta, \lambda)}$ defined in the following way.
Consider in the half plane $\mathbb R^2_{(p,q)}\cap \{p>0\}$ the
trace of curve $L_0$ and the embedding of $\mathbb R^2_{(p,q)}$
into $\mathbb R^2_{(\delta,p,q)}$ as part of the plane $\delta
=0$. Let $Z_1$ be the cylinder in $\Phi(V^+)$ over $L_0$, with the
axis parallel to $\delta$ axis of height $\delta_0$. Then for
sufficiently small $\delta_0$ the ``blown-up horn''
\[
Z=\Phi(\Sigma_1\cap V^+)
\]
is diffeomorphic to $Z_1$. The diffeomorphism taking $Z$ into
$Z_1$ preserve the foliation $\delta=$const, is $C^1$-smooth in
$\delta$, and its difference from the identity map on the fiber
$\delta=$const is of the order $O(\delta)$.
\end{theorem}

\section{Characterization of surface $\Sigma_2$}
\label{sec:05}

Recall that $\Sigma=\Sigma_1\cup \Sigma_2$ is the bifurcation
surface for limit cycles, that is, the surface in the parameter
space on which the number of cycles change due to splitting and
disappearance of multiple cycles or, what is the same, the surface
where the number of roots of equation \eqref{eq:15} change.

Consider the function $\psi$ defined by
\[
\psi (y,\delta,p,q)=-py+q
-f\left(-\frac{1}{p}f(y)+\frac{q}{p}\right)-\tilde{r}_2(y,\delta,p,q).
\]
The surface $\Sigma$ in the space of parameters $\mathbb
R^3_{(\delta,p,q)}$ is the projection on $\mathbb
R^3_{(\delta,p,q)}$ of the manifold determined by the system
\[
\begin{array}{rcl}
\psi (y,\delta,p,q) & = & 0, \\
\psi'_y (y,\delta,p,q) & = & 0.
\end{array}
\]

Let $\varphi_\delta (y,p,q)=\psi(y,\delta,p,q)$. Note that
$\displaystyle\varphi_0(y,p,q)=-py+q
-f\left(-\frac{1}{p}f(y)+\frac{q}{p}\right)$. We want to study the
apparent contour $\Lambda_0$ of the surface $\varphi_0=0$, i.e.
the projection on $\mathbb R^2_{(p,q)}$ of the curve determined by
the system
\[
\begin{array}{rcl}
\varphi_0 (y,p,q) & = & 0, \\
(\varphi_0)'_y (y,p,q) & = & 0.
\end{array}
\]
The curve determined by the system above is said to be the {\it
horizon} of the surface $\varphi_0=0$. Note that, the curve $L_0$
defined by \eqref{eq:18} belongs the apparent contour of surface
$\varphi_0=0$. In fact, the curve $(y,f'(y), f'(y)y+f(y))$ is the
parametrization of a piece of the horizon of $\varphi_0 =0$,
correspondent to surface $\Sigma_1$ and so $(f'(y), f'(y)y+f(y))$
is a piece of apparent contour of $\varphi_0=0$.
\begin{proposition}
\label{pro:02} {For} a generic function $f$ the set
\[
A=\left\{ y:\;\frac{\partial^3\varphi_0 }{\partial y^3}(y,f'(y),
f'(y)y+f(y))=0\right\}
\]
is finite and for all $y\in A$
\[
\frac{\partial^5\varphi_0 }{\partial y^5}(y,f'(y),
f'(y)y+f(y))\neq 0.
\]
Moreover, the determinant of the Jacobian matrix of the map
\linebreak$((\varphi_0)'_y (y,p,q),(\varphi_0)'''_{yyy} (y,p,q))$
with respect the parameters $p$, $q$ is different from zero at
points $(y,f'(y), f'(y)y+f(y))$ with $y\in A$.
\end{proposition}
\begin{proof}
The proof follows from Thom's  Transverslity Theorem. Details are
given for the first part; the second one follows analogously.

 We have that
\[
\varphi_0 (y,f'(y), f'(y)y+f(y))=\frac{\partial \varphi_0
}{\partial y}(y,f'(y), f'(y)y+f(y))=0,
\]
\[
\frac{\partial^2\varphi_0 }{\partial y^2}(y,f'(y), f'(y)y+f(y))=0,
\]
\[
\frac{\partial^3\varphi_0 }{\partial y^3}(y,f'(y),
f'(y)y+f(y))=-2f'''(y)+3\frac{f''(y)^2}{f'(y)},
\]
\[
\frac{\partial^4\varphi_0 }{\partial y^4}(y,f'(y),
f'(y)y+f(y))=\frac{f''(y)}{f'(y)}\left(-2f'''(y)+3\frac{f''(y)^2}{f'(y)}\right),
\]
and
\[
\frac{\partial^5\varphi_0 }{\partial y^5}(y,f'(y),
f'(y)y+f(y))=-2f^{(5)}(y)+\frac{15}{f'(y)}f^{(4)}(y)f''(y)
\]
\[
-\frac{5}{f'(y)^2}f'''(y)f''(y)^2-\frac{10}{f'(y)}f'''(y)^2.
\]
Note that, $\displaystyle\frac{\partial^3\varphi_0 }{\partial
y^3}(y,f'(y), f'(y)y+f(y))=0$ implies that $\displaystyle
f'''(y)=\frac{3f''(y)^2}{2f'(y)}$,
$\displaystyle\frac{\partial^4\varphi_0 }{\partial y^4}(y,f'(y),
f'(y)y+f(y))=0$  and
\[
\frac{\partial^5\varphi_0 }{\partial y^5}(y,f'(y),
f'(y)y+f(y))=-2f^{(5)}(y)
\]
\[
+\frac{15f''(y)}{f'(y)}\left(f^{(4)}(y)-\frac{2f''(y)^3}{f'(y)^2}\right).
\]
Now, consider the $5$-jet
\[
j^5f:\mathbb R\rightarrow J^5(\mathbb R,\mathbb R)
\]
of $f$. The space $J^5(\mathbb R,\mathbb R)$ may be identified
with $\mathbb R^7$ and the jet $j^5f$ with the map
\[
y\mapsto (y,f(y),f'(y),f''(y),f'''(y),f^{(4)}(y),f^{(5)}(y)),
\]
from $\mathbb R$ into $\mathbb R^7$; then to prove the proposition
we apply Thom's Transversality Theorem (see \cite{MD}) to the
submanifolds of codimensions $1$ and $2$ in $\mathbb R^7$
consisting of elements of the form $(y,a,b,c,d,e,g)$ with
$d=3c^2/(2b)$ and $d=3c^2/(2b)$, $-2g+15c/b(e-2c^3/b^2)=0$
respectively. Denoting the respective submanifolds by $V$ and $W$,
we have that for a generic function $f$, $j^5f^{-1}(V)$ is a
discrete set and $j^5f^{-1}(W)$ is empty.
\end{proof}

Now we will characterize the shape of the apparent contour in the
neighborhood of a point $y_0\in A$, where $A$ is the set defined
in Proposition \ref{pro:02}. These points are called
$2$-{codimension flips}. Without loss of generality, we can
suppose that $y_0=0$. Hence, from Proposition \ref{pro:02} and by
standard theory about flip bifurcation (see \cite{BV} and
\cite{JG} for more details), we have that for a generic function
$f$ after a conjugacy by a local change of coordinates $\varphi_0$
takes the form
\[
\varphi_0(y,p,q)=py+qy^3-y^5+O_{(p,q)}(\|y\|^6).
\]
Near the origin the surface $\varphi_0=0$ has the same shape as
\linebreak$G(y,p,q)/y=0$, where $G(y,p,q)=py+qy^3-y^5$. In fact,
there exists a $C^1$ local diffeomorphism mapping the apparent
contour $\Lambda_0$ of $\varphi_0=0$ onto the corresponding one
for $G=0$. Thus, we have the following result.

\begin{proposition}
\label{pro:03} For a generic function $f$, the shape of the
apparent contour $\Lambda_0$ of $\varphi_0=0$, in a neighborhood
of a point which belongs to set $A$, is given by Figure
\ref{figura2}.
\end{proposition}

\begin{figure}[ptb]
\begin{center}
\psfrag{X}[l][B]{$y$} \psfrag{P}[l][B]{$p$}
\psfrag{G}[l][B]{$\Lambda_0$} \psfrag{F}[l][B]{$\varphi_0 =0$}
\psfrag{Q}[l][B]{$q$} \psfrag{L}[l][B]{ $L_0$}
\includegraphics[height=2in,width=2in]{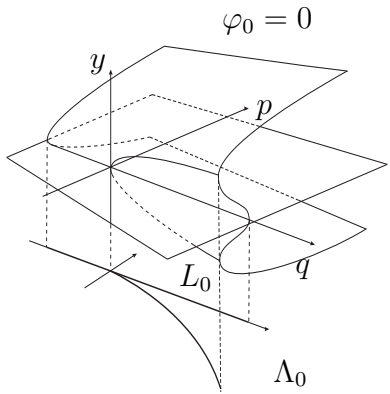}
\end{center}
\caption{Flip bifurcation of codimension $2$.} \label{figura2}
\end{figure}

The next result gives the shape of apparent contour $\Lambda_0$ of
the surface $\varphi_0=0$.

\begin{proposition}
\label{pro:04} {For} a generic function $f$, the apparent contour
$\Lambda_0$ of $\varphi_0=0$ is a curve having as singularities a
finite number of ordinary cusps and codimension-$2$ flips, its
self-intersections are transversal and occur only at smooth arcs.
Moreover, $\Lambda_0$ does not have more than one
self-intersections at a unique point, there are no
self-intersections at the endpoints of $\Lambda_0$ and they are
not singularities.
\end{proposition}
\begin{proof}
The proof of this proposition is a straightforward consequence of
Multijet Transversality Theorem (see \cite{MD}). We will prove
only the first statement. The other statements are proved in an
analogous way.

From Preposition \ref{pro:03} it follows that the set of
singularities of $\Lambda_0$ contains a finite number of
codimension-$2$ flips. Now to show that the rest of singularities
of $\Lambda_0$ is constituted of a finite number of ordinary
cusps, we will do a new parametrization.

Consider the diffeomorphism $(y,p,q)\mapsto (y,p,pw+f(y))$. In the
new variables $(y,p,w)$ we have that $\varphi_0$ is written as
\begin{equation}
\label{eq:20} \varphi_0(y,p,w)=f(w)-pw-f(y)+py.
\end{equation}
Note that when $y=w$ we have the case studied in Proposition
\ref{pro:03} which correspond to curve $L_0$ given in
\eqref{eq:18}.

By \cite{BV}, to show that the rest of singularities of
$\Lambda_0$ consists on a finite number of ordinary cusps, we must
prove that for a generic function $f$, the set
\[
\mathcal{A}=\left\{(y,p,w):\; y\neq
w\;\mbox{and}\;\varphi_0=\frac{\partial \varphi_0}{\partial
y}=\frac{\partial^2 \varphi_0}{\partial y^2}=0\right\}
\]
is finite. Furthermore, if $(y,p,w)\in \mathcal A$ then
\[
\frac{\partial^3 \varphi_0}{\partial y^3}\left( \frac{\partial
\varphi_0}{\partial p}\frac{\partial^2 \varphi_0}{\partial
yw}-\frac{\partial \varphi_0}{\partial w}\frac{\partial^2
\varphi_0}{\partial yp}\right)\left|_{(y,p,w)}\right.\neq 0.
\]
In fact, as $w=(q-f(y))/p$ by \eqref{eq:20} it follows that
\begin{eqnarray}
\label{eq:21} \frac{\partial \varphi_0}{\partial y}(y,p,w) & = &
f'(w)f'(y)-p^2,
\\
\label{eq:22} \frac{\partial^2 \varphi_0}{\partial y^2}(y,p,w) & =
& f''(w)f'(y)^2-f'(w)f''(y)p
\end{eqnarray}
and
\begin{equation}
\label{eq:23} \begin{array}{ll}\displaystyle\frac{\partial^3
\varphi_0}{\partial y^3}(y,p,w)= &
f'''(w)f'(y)^3-3f''(w)f''(y)f'(y)p \\ &
+f'''(y)f'(w)p^2.\end{array}
\end{equation}
Note that if $(y,p,w)$ is a zero of $(\varphi_0)'_{y}$, then by
\eqref{eq:21} we have that $p=\sqrt{f'(w)f'(y)}$, remember that
$p>0$. Hence \eqref{eq:20}, \eqref{eq:22} and \eqref{eq:23}
becomes
\begin{eqnarray}
\label{eq:24} \varphi_0(y,p,w) & = &
f(w)-f(y)+\sqrt{f'(w)f'(y)}(y-w),
\\
\label{eq:25} \frac{\partial^2 \varphi_0}{\partial y^2}(y,p,w) & =
& f''(w)f'(y)^2-f'(w)f''(y)\sqrt{f'(w)f'(y)}
\end{eqnarray}
and
\begin{equation}
\label{eq:26}
\begin{array}{c}\displaystyle\frac{\partial^3
\varphi_0}{\partial y^3}(y,p,w)=  f'''(w)f'(y)^3
\\  -3f''(w)f''(y)f'(y)\sqrt{f'(w)f'(y)}
+f'''(y)f'(w)^2f'(y).\end{array}
\end{equation}
Now, consider the $2$-multijet of order $3$
\[
j^3_{(2)}f:\Delta_{(2)}(\mathbb R)\rightarrow J^3_{(2)}(\mathbb
R,\mathbb R)
\]
of $f$, where $\Delta_{(2)}(\mathbb R)\subset\mathbb R^2$ denote
the set of the par $(y,w)$ with $y\neq w$ and $J^3_{(2)}(\mathbb
R^2,\mathbb R)$ denote the space of $2$-multijets of order $3$ of
functions from $\mathbb R$ into $\mathbb R$ (see \cite{MD}). The
space $J^3_{(2)}(\mathbb R,\mathbb R)$ may be identified with a
subset of $\mathbb R^{10}$ and the $2$-multijet $j^3_{(2)}f$ with
the restriction of the map
\[
(y,w)\mapsto
(y,w,f(y),f(w),f'(y),f'(w),f''(y),f''(w),f'''(y),f'''(w)),
\]
from $\mathbb R^2$ into $\mathbb R^{10}$; then to prove that
$\mathcal A$ is finite and $(\varphi_0)'''_{yyy}\neq 0$ we apply
Multijet Transversality Theorem (see \cite{MD}) to the
submanifolds of codimensions $2$ and $3$ in $\mathbb R^{10}$
consisting (by \eqref{eq:24}, \eqref{eq:25} and \eqref{eq:26}) of
elements of the form $(y,w,a,b,c,d,e,g,k,l)$ with
$b-a+\sqrt{dc}(y-w)=0$, $gc^2-de\sqrt{dc}=0$ and
$b-a+\sqrt{dc}(y-w)=0$, $gc^2-de\sqrt{dc}=0$,
$lc^3-2gc\sqrt{dc}+kd^2c=0$ respectively. Denoting the respective
submanifolds by $V$ and $W$, since $\Delta_{(2)}(\mathbb R)$ has
dimension $2$, we have that for a generic function $f$,
$j^3_{(2)}f^{-1}(V)$ is a discrete set and $j^3_{(2)}f^{-1}(W)$ is
empty. In the analogous way we can prove that if $(y,p,w)\in
\mathcal A$ then $(
(\varphi_0)'_{p}(\varphi_0)''_{yw}-(\varphi_0)'_{w}(\varphi_0)''_{yp})(y,p,w)\neq
0$. This proves our claim and finishes the proof of the
proposition.
\end{proof}

Note that, by the Proposition \ref{pro:04} the apparent contour
$\Lambda_\delta$ of the surface $\varphi_\delta=0$ is
diffeomorphic to $\Lambda_0$. This follows from the fact that the
curve $\Lambda_0$ is structurally stable, i.e. the curve
$\Lambda_0$ does not change its topological structure by small
perturbations. Hence the next result follows directly (see Lemma
$8$, page $197$ of \cite{KS} for a similar result).
\begin{lemma}
The apparent contour $\Lambda_\delta$ of the surface
$\varphi_\delta=0$ for a generic function $f$ is diffeomorphic to
$\Lambda_0$ and the diffeomorphism \linebreak smoothly depends on
$\delta$ as $\delta\rightarrow 0$.
\end{lemma}

Thus, as in the previous section we have a similar theorem that
characterize the shape of the surface $\Sigma_2$.

\begin{theorem}
\label{the:02} The surface $\Sigma_2$ is a horn in $\mathbb
R^3_{(\epsilon,\delta,\lambda)}$ defined in the following way.
Consider in the half plane $\mathbb R^2_{(p,q)}\cap \{p>0\}$ the
trace of curve $\Lambda_0\setminus L_0$ and the embedding of
$\mathbb R^2_{(p,q)}$ into $\mathbb R^2_{(\delta,p,q)}$ as part of
the plane $\delta =0$. Let $Z_2$ be the cylinder in $\Phi(V^+)$
over $\Lambda_0\setminus L_0$ with the axis parallel to $\delta$
axis of the height $\delta_0$. Then for sufficiently small
$\delta_0$ the ``blown-up horn''
\[
Z=\Phi(\Sigma_1\cap V^+)
\]
is diffeomorphic to $Z_2$. The diffeomorphism taking $Z$ into
$Z_1$ preserve the foliation $\delta=$const, is $C^1$-smooth in
$\delta$, and its difference from the identity map on the fiber
$\delta=$const is if order $O(\delta)$.
\end{theorem}

\section{Bifurcation diagram for the lips on a M\"obius band}
\label{sec:06}

In this section we give the precise local description of the
bifurcation diagram for the equation \eqref{eq:15}. As it was
explained in Section \ref{sec:01}, this diagram consists of four
parts, $\Sigma_1$, $\Sigma_2$, $\Sigma_+$ and $\Sigma_-$.

\begin{theorem}
\label{the:08} Let $\displaystyle \varphi_\delta(y,p,q)$ be the
map defined in Section \ref{sec:05}.

\begin{enumerate}

\item The surface $\Sigma=\Sigma_1\cup\Sigma_2$ is a horn in
$\mathbb R^3_{(\epsilon,\delta,\lambda)}$ defined in the following
way. Consider in the half plane $\mathbb R^2_{(p,q)}\cap \{p>0\}$
the apparent contour $\Lambda_0$ of surface $\varphi_0=0$ and the
embedding of $\mathbb R^2_{(p,q)}$ into $\mathbb
R^3_{(\delta,p,q)}$ as part of the plane $\delta =0$. Let
$Z_{\Lambda_0}$ be the cylinder in $\Phi(V^+)$ over $\Lambda_0$
with the axis parallel to $\delta$ axis of height $\delta_0$. Then
for sufficiently small $\delta_0$ the ``blown-up horn''
\[
Z=\Phi(\Sigma\cap V^+)
\]
is diffeomorphic to $Z_{\Lambda_0}$. The diffeomorphism taking $Z$
into $Z_{\Lambda_0}$ preserve the foliation $\delta=$const, is
$C^1$-smooth in $\delta$, and its difference from the identity map
on the fiber $\delta=$const is of order $O(\delta)$ (see
Figure~\ref{figura03}).

\item Consider in the half plane $\mathbb R^3_{(\delta,p,q)}\cap\{
p>0\; \mbox{and}\; \delta=\mbox{const}\}$ the curves
$l^\pm_\delta$ determined by the equations $\varphi_\delta(\pm
1,p,q)=0$. In the coordinates $(\delta, p, q)$, the boundary
surfaces $\Sigma_{+}$ and $\Sigma_{-}$ are regular surfaces
diffeomorphic to cylinders in $\Phi(V^+)$ over $l^\pm_0$ with the
axis parallel to $\delta$ axis of height $\delta_0$ (for
sufficiently small $\delta_0$). This diffeomorphism preserve the
foliation $\delta=$const, is $C^1$-smooth in $\delta$, and its
difference from the identity map on the fiber $\delta=$const is of
order $O(\delta)$.

\item The intersection of the boundary of $\Sigma$ with the layer
$0< \delta < \delta_{0}$ belongs to $\Sigma_{+}$ and $\Sigma_{-}$.
At points of this intersection, the surface $\Sigma$ is tangent to
either $\Sigma_{+}$ or $\Sigma_{-}$.

\end{enumerate}
\end{theorem}

\begin{proof} Let $\displaystyle \varphi_\delta(y,p,q)$ be the map defined
in the previous section.

The first statement of theorem follows from Theorems \ref{the:03}
and \ref{the:02}.

{For} to prove the second assertion of the theorem we consider
only the surface $\Sigma_{+}$. The study of the surface
$\Sigma_{-}$ is exactly analogous. The surface $\Sigma_{+}$ is
described by the equation
\[
\varphi_\delta(1,p,q)=0.
\]
Now, consider the diffeomorphism $(y,p,q)\mapsto (y,p,pw+f(y))$.
In the new variables $(y,p,w)$ we have that the equation
$\varphi_0(1,p,q)=0$ is written as
\begin{equation}
\label{eq:29} f(w)-pw-f(1)+p=0.
\end{equation}
We can write equation \eqref{eq:29} as
\[
p=\int^{1}_{0}f'(1+s(w-1))ds.
\]
Thus the curve determined by equation \eqref{eq:29} is the graphic
of a function of the forme $p=g(w)$. Hence the curve $l_0^+$ is
regular and has no self-intersections. Therefore for a fixed
$\delta$ ($0<\delta<\delta_0$), the curve determined by the
equation $\varphi_\delta(1,p,q)=0$ is regular and is diffeomorphic
to $l_0^+$. This prove the second statement of the theorem.

The values of the parameters for which the equation of limit
cycles has a multiple root at $y=\pm 1$ belong both to the
boundary of $\Sigma$ and to one of surfaces $\Sigma_\pm$. Let us
show that at that point the surfaces are tangent to each other.
Consider the case $y=1$, the other case is analogous. The previous
claim is equivalent to saying that the surface $Z$ is tangent to
$\Phi (\Sigma_+)$. It is sufficient to show that the apparent
contour $\Lambda_\delta$ ($0<\delta <\delta_0$) of the surface
$\varphi_\delta=0$ is tangent to the curve $\Phi (\Sigma_+)\cap
\{\delta=\mbox{const}\}$ at the endpoint of $\Lambda_\delta$
corresponding to $y=1$. In the variables $(y,p,w)$, the curve
$\Phi (\Sigma_+)\cap \{\delta=\mbox{const}\}$ is described by the
equation
\[
\varphi_\delta(1,p,w)=0.
\]
Therefore everywhere on such curve we have
\[
\frac{dp}{dw}=\frac{(\varphi_\delta)'_w(1,p,w)}{(\varphi_\delta)'_p(1,p,w)}.
\]
Note that $(\varphi_\delta)'_p(1,p,w)\neq 0$. On the similar way,
in the variables $(y,p,w)$, the smooth parts of $\Lambda_\delta$
are graphs of smooth functions of the form $p=h(w)$, where
\[
h'(w)=\frac{(\varphi_\delta)'_w(y,p,w)}{(\varphi_\delta)'_p(y,p,w)}.
\]
This means that the two curves touch each other, hence the
surfaces $\Sigma$ and $\Sigma_+$ are tangent. Thus the proof of
this theorem is complete. 
\end{proof}

\begin{figure}[ptb]
\begin{center}
\psfrag{D}[l][B]{$\delta$} \psfrag{Z}[l][B]{$Z_{\Lambda_0}$}
\psfrag{Q}[l][B]{$q$} \psfrag{P}[l][B]{$p$} \psfrag{F}[l][B]{$\Phi
(\Sigma\cap V^{+})$} \psfrag{G}[l][B]{$\Lambda_0 $}
\includegraphics[height=2in,width=3in]{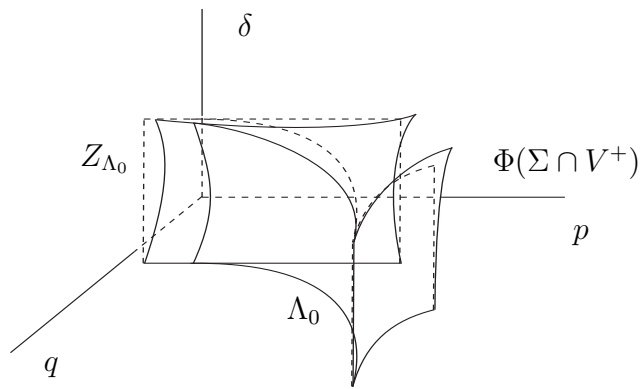}
\end{center}
\caption{The bifurcation horn.} \label{figura03}
\end{figure}

\begin{remark} The local behavior of the number of limit cycles in the
bifurcation diagram established in theorem \ref{the:08} is made
more explicit is as follows.

 We can consider in the parameter
space $(\delta,p,q)$ only the curves determined by the
intersection of the bifurcation surfaces with the plane
$\delta=$constant, i.e the curves $\Lambda_\delta$ and
$l_\delta^\pm$.

 Near a cuspidal point of $\Lambda_\delta$ that
corresponds to a root of multiplicity $3$ of the equation
$\varphi_\delta=0$, the number of limit cycles decreases by $2$
when passing from the inside  to the outside of the cusp (see
Figure~\ref{figura3} $(a)$).

Near the endpoints of the curve $\Lambda_\delta$, i.e. at the
intersection points of $\Lambda_\delta$ with  the curves
$l_\delta^\pm$ that correspond to a double root at an extreme
point $y=\pm 1$, there are three local connected components of
$\mathbb R^2_{(p,q)}\setminus \{\Lambda_\delta\cup
l_\delta^\pm\}$; one is a piece of half space, the other is
locally convex and the third one is a ``thin" horn. The number of
limit cycles decreases by $1$  when we move from the ``thin"
component to the half space as well as  when we go from the half
space  to the locally convex component (see Figure~\ref{figura3}
$(c)$).

Near a flip point of $\Lambda_\delta$ that corresponds to a root
of multiplicity $5$ of  $\varphi_\delta=0$, the behavior is
analogous to the one just described (see Figure~\ref{figura3}
$(b)$).
\end{remark}

\begin{figure}[ptb]
\begin{center}
\psfrag{N}[l][B]{$n$} \psfrag{N1}[l][B]{$n+1$}
\psfrag{N2}[l][B]{$n+2$} \psfrag{L}[l][B]{$l_\delta^\pm$}
\psfrag{Ld}[l][B]{$\Lambda_\delta$} \psfrag{L1}[l][B]{$L_\delta$}
\psfrag{Ld1}[l][B]{$\Lambda_\delta\setminus L_\delta$}
\psfrag{A}[l][B]{$(a)$} \psfrag{B}[l][B]{$(b)$}
\psfrag{C}[l][B]{$(c)$}
\includegraphics[height=1in,width=3in]{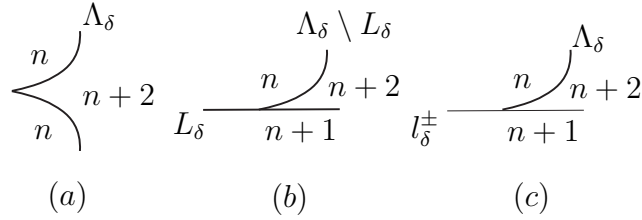}
\end{center}
\caption{Bifurcation of limit cycles near cuspidal singularities,
flip points and endpoints of $\Lambda_\delta$. Here $n$ indicate
the number of limit cycles.} \label{figura3}
\end{figure}

\section{Finite Cyclicity on a M\"obius lips}
\label{sec:fc}

{\it Cyclicity} of a polycycle in a family of vector fields
depending of parameters is the maximal number of limit cycles
generated by this polycycle and corresponding to a parameter value
close to the one with the polycycle.

Let $f$ and $g$ be two smooth real functions. We say that $f$ and
$g$ are {\it affine equivalent} if there exist an affine function
$\alpha(x)=b_1x+b_2$ such that $f\circ \alpha=\alpha^{-1}\circ g$.

The study of the cyclicity of the lips in the M\"obius band is
similar the study of two subsequent lips, i.e. with four
saddle-nodes. This can be seen by considering the double
orientable covering of the M\"obius band. In \cite{DIR} the
cyclicity of a polycycle that belongs to two subsequent lips is
determined. We can use a similar idea to obtain the following
result.

\begin{theorem}
\label{the:l}
Let $X_\mu$, $\mu \in \mathbb R^3$, be a generic $C^{\infty}$
$3$-parameter family of vector fields in a M\"obius band $M^2$,
such that $X_0$ has a set of polycycles of lips type. Consider the
family of diffeomorphisms $f_\nu$ of class $C^k$ defined in \eqref{eq:01}.
Suppose that $f_{0}$ and $f^{-1}_0$ satisfies the following
generic conditions at some point $y_0$:
\begin{itemize} \item[(i)] the jets $J^n_{y_0}f_0$, $J^n_{y_0}f^{-1}_0$ of $f_0$ and $f^{-1}_0$ at the
point $y_0$ are nonaffine maps to some order $n\leq k$,
\item[(ii)] the jets $J^n_{y_0}f_0$, $J^n_{y_0}f^{-1}_0$ of $f_0$
and $f^{-1}_0$ at the point $y_0$ are not affine equivalent
through orientation reversing affine maps for some order $n\leq
k$.
\end{itemize}
Then the polycycle of the lips passing through a point $y_0$ has
cyclicity $\leq n$, where $n\leq k$ is the minimal order of the
jet of the two functions $f_0$ and $f_0^{-1}$ at the point $y_0$
on which we can check the genericity conditions $(i)$ and $(ii)$.
\end{theorem}
\begin{proof}
By \eqref{eq:11}, we can write the equation that determine the
limit cycles in the following form
\[
f_\nu^{-1}(-py+q)=-\frac{1}{p}f_\nu(y)+\frac{q}{p}.
\]
Now, we define the displacement map $V_\nu:\Gamma^+_1\rightarrow
\mathbb R$ by $V_\nu(y)=\displaystyle
f_\nu^{-1}(-py+q)+\frac{1}{p}f_\nu(y)-\frac{q}{p}$.

Without loss of generality we can suppose that $y_0=0$. Then, we
will prove that for some $n\leq k$, $V_\nu^{(n)}(0)\neq 0$ for all
$\nu\neq 0$ small enough.

Consider the map
\[
V_{(p,q)}(y)=\displaystyle
f_0^{-1}(-py+q)+\frac{1}{p}f_0(y)-\frac{q}{p}
\]
with $p>0$. By the hypotheses of the theorem the system of
equations
\[
V_{(p,q)}^{(m)}(0)=\displaystyle
(-p)^m(f_0^{-1})^{(m)}(q)+\frac{1}{p}f_0^{(m)}(0)=0,\;\;\; m=1,\ldots, k,
\]
has no solutions on the variables $p$, $q$. Therefore, there exist
$n\leq k$ such that $V_{(p,q)}^{(n)}(0)\neq 0$ for all $p>0$ and
$q$ sufficiently small. Hence $V_{\nu}^{(n)}(0)\neq 0$ for all
$\nu\neq 0$ small enough. This implies that in a small
neighborhood of the polycycle passing by $y_0$ at most $n$ limit
cycles can bifurcate from it. The theorem is proved.
\end{proof}



The proof of Theorem \ref{the:l} is inspired in Theorem $3.3$ and Main Lemma
$3.4$ of \cite{DIR}.

\end{document}